\documentclass[a4paper, 11pt]{amsart}

\usepackage{amssymb}
\usepackage{amsmath}
\usepackage{amsthm}
\usepackage{amsfonts}
\usepackage{enumerate}
\usepackage{hyperref}
\usepackage{tikz}
\usepackage[left=2.5cm,right=2.5cm]
{geometry}

\theoremstyle{plain}
\newtheorem{theorem}{Theorem}[section]
\newtheorem{lemma}[theorem]{Lemma}
\newtheorem{prop}[theorem]{Proposition}

\theoremstyle{definition}
\newtheorem{definition}[theorem]{Definition}

\theoremstyle{remark}

\newcommand{\mb}[1]{\mathbb{#1}}

\newcommand{\mfr}[1]{\mathfrak{#1}}

\title{On a theorem of S\'ark\"ozy for difference sets and shifted primes}

\author{ruoyi wang}
\address{Mathematical Institute, University of Oxford, Radcliffe Observatory Quarter, Woodstock Road, Oxford OX2 6GG, England}
\email{wangr@maths.ox.ac.uk}
\thanks{The study of the author is supported by a Clarendon Scholarship of the University of Oxford, and a Jason Hu Scholarship of Balliol College.}

\begin{document}

\begin{abstract}
We show that if the difference of two elements of a set $A \subseteq [N]$ is never one less than a prime number, then $|A| = O (N \exp (-c (\log N)^{1/3}))$ for some absolute constant $c>0$.
\end{abstract}

\maketitle

\section{Introduction}
In 1978, S\'ark\"ozy published a series of papers~\cite{Sar78a, Sar78b, Sar78} studying difference sets of sequences of integers. One of his results concerns shifted prime numbers. In this article, we shall prove the following.
\begin{theorem}\label{main result}
Let $N$ be a positive integer and let $[N]$ denote the set $\{1, ..., N\}$. Suppose that the difference of any two elements of a set $A \subseteq [N]$ is never one less than a prime number, then $|A| = O( N \exp(-c (\log N)^{1/3}) )$ where $c>0$ is absolute.
\end{theorem}
S\'ark\"ozy~\cite{Sar78} established the first explicit bound of $|A|$ using the Hardy--Littlewood method and an iteration argument. He showed\footnote{Here we use $\log_k$ to denote the $k$-fold logarithm.} that $|A| = O( N \exp(- (2 + o(1)) \log_3 N))$. Subsequently, the bound $|A| = O(N \exp ( - \omega(N) \log_3 N))$, where $\omega(N)$ tends to infinity as $N \to \infty$,\footnote{More precisely, $\omega(N) \sim c\log_5 N$ for some $c > 0$.} was obtained by Lucier~\cite{Luc08} who improved S\'ark\"ozy's argument using the methods of Pintz, Steiger and Szemer\'edi~\cite{PSS88}. By exploiting a dichotomy depending on whether an exceptional zero of Dirichlet $L$-functions occurs or not, Ruzsa and Sanders~\cite{RS08} proved that $|A| = O( N \exp(-c (\log N)^{1/4}) )$, where $c>0$ is an absolute constant. 

Our key ingredient is a new major arc estimate. The underlying motivation of building the estimate is to reduce the size of the error terms, which plays a crucial role in the method of Ruzsa and Sanders. To achieve this, we need to control the contribution of the zeros of $L$-functions and avoid additional errors caused by the length of the major arcs.

The main tool for bounding the error caused by zeros of $L$-functions is the exceptional zero repulsion~\cite[Principle 3, Section 18.1]{IK04}. Roughly speaking, it compensates for the reduced size of the main term when an exceptional zero occurs. In fact, on the way of proving Linnik's theorem, one obtains a better error term in the prime number theorem in arithmetic progressions (see Iwaniec and Kowalski~\cite[Equations (18.87) and (18.89)]{IK04}) than the one used by Ruzsa and Sanders~\cite[Proposition 4.7]{RS08}.

The error terms would be too large for the combinatorial arguments should we integrate the error term of the prime number theorem directly. To deal with this issue, we shall exploit all terms involving zeros with large real parts and small imaginary parts using the classical explicit formula to improve the major arc estimate.

The article is organised as follows. We prove the major arc estimate in Section 4. The details of the set-up and the number-theoretic results used in the proof are given in Section 3. In Section 5, we use our improved major arc estimate to modify the iteration scheme of Ruzsa and Sanders and get our result.

\section{Notation}
Let $f \in \ell^1 (\mb{Z})$. The Fourier transform of $f$ is defined as the function $\widehat{f} : \mb{T} \to \mb{C}$ given by
$$\widehat{f} (\theta) := \sum_{x \in \mb{Z}} f(x) e(-x\theta),$$
where $e(\theta) := e^{2 \pi i \theta}.$ We define the convolution of two functions $f,g \in \ell^1 (\mb{Z})$ by
$$f * g (x) := \sum_{y \in \mb{Z}} f(x-y) g(y).$$

Let $Q$ be a positive parameter which will be clarified in the relevant context. For positive integers $a$ and $q \leq Q$, we define
\begin{equation}\label{major arcs notation}
\mfr{M}_{a,q} := \{ \theta \in \mb{T}: |\theta - a/q| \leq 1/(qQ) \} \;\;\text{and}\;\; \mfr{M}^*_{q} := \bigcup_{\substack{1 \leq a \leq q \\ (a,q) = 1}} \mfr{M}_{a,q}.
\end{equation}
We shall identify the torus $\mb{T}$ with an interval of length $1$ when necessary. 

We use capital letter $C$ with subscripts to denote absolute constants which tend to be large, and $c$ with subscripts to denote absolute constants which are small (and at least less than $1$).

\section{Zeros of Dirichlet $L$-functions}\label{zeros}
In this section, we focus on the number-theoretic results we need. We first show an easy consequence of various results concerning the location of zeros of Dirichlet $L$-functions, showing that Definition~\ref{dichotomy} is exhaustive. Next we list a couple of standard bounds related to the contribution of terms which involve the zeros of $L$-functions.

Let $\chi$ be a Dirichlet character of modulus $q$. We define
$$\psi (x, \chi) := \sum_{n \leq x} \chi(n) \Lambda (n),$$
where $\Lambda$ is the von Mangoldt function. The classical way of studying $\psi (x, \chi)$ is to understand the zeros of the corresponding $L$-function $L(s,\chi)$, which is defined as the analytic continuation of the function
$$\sum_{n = 1}^\infty \dfrac{\chi(n)}{n^s}, \;\; \Re({s}) >1.$$
In fact, to study the location of the zeros of $L$-functions in the critical strip $0 < \Re(s) <1$, it suffices to study the zeros of $L(s,\chi)$ for all primitive characters. This can be seen by using the Euler product expansion. Let $\chi$ be a character of modulus $q$ which is induced by a primitive character $\chi_1$ and let $\chi'$ be the principal character of modulus $q$. It follows from the definition of induced characters that $\chi = \chi_1 \chi'$, and so 
\begin{equation}\label{non-primitive}
L(s, \chi) = \prod_{p \mid q} (1 - \chi_1 (p) p^{-s}) L(s, \chi_1)\;\; \text{for}\;\; \Re(s) >1.
\end{equation}
Therefore, by analytic continuation, $L(s,\chi)=0$ if and only if $L(s, \chi_1) =0$ in the region $\Re(s) >0$. 

For any Dirichlet character $\chi$ and $T \geq 1$, we define
\begin{equation}\label{zero region}
Z(\chi; T) := \{\rho: L(\rho, \chi) = 0, \Re({\rho}) \geq 1/2, |\Im({\rho})| \leq T \}
\end{equation}
and
$$Z(q;T) := \bigcup_{\chi \; (\mathrm{mod} \; q)} Z(\chi ; T).$$
We treat a zero with multiplicity $m$ as $m$ elements in the zero sets above, and we use $|Z(\chi; T)|$ to denote the cardinality of this set.

The following lemma follows from known results about zeros of Dirichlet $L$-functions.
\begin{lemma}\label{dichotomy lemma}
There are positive absolute constants\footnote{We need $C_1 \geq 10$ due to the choice of minor arcs in the final section.} $c_1, c_2$ and $C_1 \geq 10$ such that for any $D \geq 2$ and $T \geq 1$, the following assertions hold. 

Suppose there exists a primitive character $\chi_D$ such that $\chi_D$ has modulus $d_D \leq D$ and $L(s,\chi_D)$ has a zero $\beta_D$ in the region
\begin{equation}\label{lem 3.1 region}
\Re(s) \geq 1 - \dfrac{c_1}{C_1 \log(DT)}, \;\; |\Im(s)| \leq T.
\end{equation}
Then 
\begin{enumerate}[(i)]
\item the zero $\beta_D$ is real and simple, and it is the only zero of $L(s,\chi_D)$ in the region (\ref{lem 3.1 region});
\item there does not exist any other primitive character $\chi$ of modulus $q \leq D^{C_1}$ such that $L(s,\chi)$ has a zero in the region (\ref{lem 3.1 region});
\item (exceptional zero repulsion) for any $d_D \mid d$, all other zeros in $Z(dq; T)$ are in the region
$$\Re(s) \leq 1 -\dfrac{ c_{2} |\log ( ( 1 - \beta_D) \log(dqT))|}{\log (dqT)}, \; |\Im(s)| \leq T.$$
\end{enumerate}

\end{lemma}
\begin{proof}
The result follows from Principle 1 Chapter 18 (zero-free region), Principle 3 Chapter 18 (a quantitative version of exceptional zero repulsion) and Theorem 5.28 of Iwaniec and Kowalski~\cite{IK04}.
\end{proof}

We shall split into two different cases depending on whether a possible exceptional zero, as defined in Definition~\ref{dichotomy} below, exists or not. More precisely, our set-up involves two parameters, one of which controls the modulus of the exceptional primitive character and another the height of the rectangle which contains the zeros we need to consider, and they jointly quantify our notion of being exceptional. 

\begin{definition}\label{dichotomy}
Let $C_1$ and $c_1$ be the constants from Lemma~\ref{dichotomy lemma}. Let $D \geq 2$ and $T \geq 1$. 
We say that $(D,T)$ is {\it exceptional} if there exists a unique primitive character $\chi_D$ such that $\chi_D$ has modulus $d_D \leq D$, and $L(s,\chi_D)$ has a zero $\beta_D$ which is real and simple and satisfies $\beta_D \geq 1 - c_1/ (C_1 \log (DT))$. We call $\chi_D$ the {\it exceptional character} and $\beta_D$ the {\it exceptional zero}. Otherwise, we say that $(D,T)$ is {\it unexceptional}.
\end{definition}

By the truncated explicit formula, one has the following estimates, see Iwaniec and Kowalski~\cite[Section 18.4; see also Proposition 5.25]{IK04}.\footnote{Here we have an extra restriction $\Re({\rho}) \geq 1/2$ compared to the explicit formula given in Iwaniec and Kowalski~\cite[Proposition 5.25]{IK04}. The reason is that we have absorbed the error caused by the zeros whose real parts are smaller than $1/2$ into the error term.}
\begin{prop}\label{truncated explicit formula}
Let $q$ be a positive integer and $x > 0$. For any character $\chi$ of modulus $q$ and any $1 \leq T \leq x^{1/4}$, one has
$$\sum_{n \leq x} \Lambda(n) \chi(n) = x 1_{\chi'}(\chi) - \sum_{\rho \in Z(\chi; T)} \dfrac{x^\rho}{\rho} + O \left( \dfrac{x \log^2 (qx)}{T} \right).$$
Here $1_{\chi'} (\chi) = 1$ if $\chi$ is the principal character, and $1_{\chi'}(\chi) = 0$ otherwise.
\end{prop} 

Later on, we need to bound the contribution of the zeros in the region $Z(\chi;T)$. The proposition below is introduced for this purpose, and it can be shown by using arguments from Iwaniec and Kowalski~\cite[Section 18.4]{IK04}. It turns out that in the unexceptional case, the zero density estimate~\cite[Principle 2 Chapter 18]{IK04} will be strong enough to produce the desired bound. In the exceptional situation, this is no longer the case, since the size of the main term could be reduced due to the term which contains the exceptional zero. To compensate for this, we use the exceptional zero repulsion to deduce a stronger bound on the contribution of other zeros.

\begin{prop}\label{Explicit formula}
There exist absolute constants $C_2$ and $c_{3}$ such that for all $x, D, T \geq 2$ satisfying $x> (DT)^{C_2}$, we have the following.
\begin{enumerate}[(i)]
\item If $(D,T)$ is unexceptional, then for any $q,d \geq 1$ satisfying $dq \leq D$, we have
$$\sum_{\rho \in Z(dq; T)} {|x^{\rho - 1}|} = O \left( \exp\left( -c_{3} \dfrac{\log x}{\log (DT)} \right) \right).$$
\item If $(D,T)$ is exceptional, then for any $q, d \geq 1$ satisfying $dq \leq D^{C_1}$ and $d_D \mid d$, we have
$$\sum_{\substack{\rho \in Z(dq; T) \\ \rho \neq \beta_D }} {|x^{\rho - 1}|} = O\left( (1 - \beta_D)\log(dqT) \exp\left( -c_{3} \dfrac{ \log x}{\log(dqT)} \right) \right).
$$
Here, $C_1$ is the constant from Definition~\ref{dichotomy}.
\end{enumerate}
\end{prop}

\section{Major arc estimates}\label{major arc estimate section}
For any positive integers $N,d$, let
\begin{equation}\label{def}
F_{N,d} (n) :=  \Lambda(dn + 1) 1_{[N]} (n),
\end{equation}
where $1_{[N]}$ is the characteristic function of the set $[N]$.

Such functions are used to detect primes and prime powers in arithmetic progressions, which will in turn provide the desired structure in the iteration scheme in the upcoming section. The iteration is done by an energy increment argument, and we shall need appropriate estimates of the Fourier transform of $F_{N,d}$ to effect this.

The goal of this section is to prove the following result.
\begin{prop}[Major arc estimates]\label{major arc estimate}
There exist positive absolute constants $C_3$ and $c_4$ such that the following holds. 

Let $T, D \geq 2$ and let $N$ be a positive integer such that $N > (DT)^{C_3}$. 

\begin{enumerate}[(1)]
\item Suppose that $(D,T)$ is unexceptional. Then for any $\delta \in [-1/2, 1/2]$ and any positive integers $a,d,q$ satisfying $(a,q)=1$ and $dq \leq D$, we have
\begin{multline*}
\left| \widehat{F_{N,d}} \left( \dfrac{a}{q} + \delta \right) \right| \leq \dfrac{2 |\widehat{F_{N,d}} (0)| }{\phi(q)}+  O\left( \dfrac{dNq}{\phi (d)\phi(q)}  \exp\left(- c_4 \dfrac{\log N}{\log (DT)} \right) \right)
+ O \left( (1 + N|\delta| )\dfrac{dqN\log^2 N}{T} \right).
\end{multline*}
We also have
$$\left| \widehat{F_{N,d}} \left( 0 \right) \right| \geq \dfrac{dN}{2\phi(d)}  - O \left( \dfrac{dN\log^2 N}{T} \right).$$

\item Suppose that $(D,T)$ is exceptional. Then for any $\delta \in [-1/2, 1/2]$ and any positive integers $a,d,q$ satisfying $(a,q)=1$, $dq \leq D^{C_1}$ and $d_D \mid d$, we have
\begin{multline*}
\left| \widehat{F_{N,d}} \left( \dfrac{a}{q} + \delta \right) \right| \leq \dfrac{2 |\widehat{F_{N,d}} (0)| }{\phi(q)}+ 
O\left( \dfrac{dNq}{\phi (d)\phi(q)} (1 - \beta_D)\log(dqT) \exp\left( - c_4 \dfrac{ \log N}{\log(dqT)} \right) \right) \\
+ O \left( (1 + N|\delta| )\dfrac{dqN\log^2 N}{T} \right).
\end{multline*}
We also have
$$\left| \widehat{F_{N,d}} \left( 0 \right) \right| \geq \dfrac{dN}{\phi(d)} \dfrac{(1 - \beta_D)\log(dT)}{4c_1} - O \left( \dfrac{dN\log^2 N}{T} \right).$$
Here $C_1$ and $c_1$ are the constants in Definition~\ref{dichotomy}.
\end{enumerate}
\end{prop}

The purpose of our first lemma is to write the Fourier transform of the function $F_{N,d}$ in a way which allows us to use known techniques related to exponential sums and $\psi(x, \chi).$ 
\begin{lemma}\label{rearrange the sum}
Let $N,a,d,q$ be positive integers and let $-1/2 \leq \kappa \leq 1/2$. One has
$$\widehat{F_{N,d}} \left( \dfrac{a}{q} + \kappa \right) =  \dfrac{1}{\phi (dq)} \sum_{\chi \; (\mathrm{mod} \; dq) } e\left(\dfrac{\kappa}{d}\right) S_{dN+1} \left( \dfrac{\kappa}{d}, \chi \right) G_{a,q,d,\chi} + O \left( (\log (dN) ) (\log q) \right), $$
where
\begin{equation}\label{sum along characters}
S_x (\delta, \chi) := \sum_{n \leq x} \Lambda(n) \chi(n) e\left( - n \delta \right)
\end{equation}
and 
\begin{equation}\label{Gauss type exponential sum}
G_{a,q,d,\chi} :=  \sum_{m =0}^{q-1} e\left( -\dfrac{am}{q} \right) \overline{\chi}(dm+1).
\end{equation}
\end{lemma}

The expansion involving Dirichlet characters helps us to reduce the task of obtaining major arc estimates to estimating sums $G_{a,q,d,\chi}$ and $S_{dN+1}(\delta, \chi)$.

It turns out that we shall only need a nontrivial bound on $G_{a,q,d,\chi}$ when $\chi = \chi'$ is the principal character. In this situation, it follows from either an application of the Ramanujan sum formula~\cite[Section 3.2, Equation (3.3)]{IK04} or a cancellation of exponential sums that
\begin{equation}\label{Gauss sum principal character}
|G_{a,q,d,\chi'}| = 1 \text{ if $(d,q) = 1$, and } G_{a,q,d,\chi'} = 0 \text{ otherwise} .
\end{equation}
For the other characters, we shall use the trivial bound $|G_{a,q,d,\chi}| \leq q$, which comes from adding the absolute value of each term in the exponential sum.

We shall also need to estimate sums of form $S_{dN+1}(\delta, \chi)$ where $\delta$ is relatively close to $0$. By partial summation and classical complex-analytic number theory, we can express the sum $S_{dN+1}(\delta, \chi)$ in terms of certain zeros of $L(s,\chi)$ up to a small error term. Consequently, we can perform integration by parts when we estimate the impact of these zeros, instead of simply integrating the absolute value of the error term caused by them. Later on, we need to apply the major arc estimate when $\delta \gg d^4 /N$ (see (\ref{major arc - 2})), and this explicit computation helps us to deal with such situations.

The aim of the next lemma is to connect the sum $S(\delta, \chi)$ with zeros of $L(s,\chi)$, which is done by partial summation. 
\begin{lemma}\label{estimate}
Let $N,q,d$ be positive integers. For any Dirichlet character $\chi$ of modulus $dq$, $-1/2 \leq \delta \leq 1/2$ and\footnote{The exponent $1/32$ is introduced purely for technical reasons. It follows from the restriction of $T$ in Proposition~\ref{truncated explicit formula} and our truncation of integral in the proof of Lemma~\ref{estimate}.} $1 \leq T \leq N^{1/32}$, we have
\begin{equation*}
 S_{dN+1} \left( \delta , \chi \right) =  \int_{N^{1/8}}^{dN+1}  \left( 1_{\chi'} (\chi) -\sum_{\rho \in Z(\chi; T)} t^{\rho-1} \right) e^{-2\pi i \delta t} dt + O\left( (1+ dN |\delta|) \dfrac{dN \log^2 (dqN)}{T} \right).
 \end{equation*}
\end{lemma}

\begin{proof}
We first deal with the case when $\delta \neq 0$. By Abel's summation formula, we have
$$S_{dN+1}\left( \delta , \chi \right) =  e^{- 2\pi i \delta (dN+1)} \sum_{n \leq dN+1} \Lambda(n) \chi(n)  \\ + 2\pi i {\delta} \int_{1}^{dN+1} \left(  \sum_{1 < n \leq t} \Lambda(n) \chi(n) \right) e^{-2\pi i t \delta} dt.$$
It follows that
\begin{multline*}
S_{dN+1} \left( \delta , \chi \right) = e^{- 2\pi i \delta (dN+1)} \sum_{n \leq dN+1} \Lambda(n) \chi(n)  \\ + 2\pi i {\delta} \int_{N^{1/8}}^{dN+1} \left(  \sum_{1 < n \leq t} \Lambda(n) \chi(n) \right) e^{-2\pi i t \delta} dt 
+ O \left( |\delta| N^{1/4} \log N \right),
\end{multline*}
where the error term is obtained by bounding the integral over $[1, N^{1/8}]$ by
$$ \left| 2\pi i \delta \int_{1}^{N^{1/8}} \left(  \sum_{1 < n \leq t} \Lambda(n) \chi(n) \right) e^{-2\pi i t \delta} dt \right| \leq 2 \pi |\delta| N^{1/8} N^{1/8} \log N =  O \left( |\delta| N^{1/4} \log N \right).$$

By substituting the expression given in Proposition~\ref{truncated explicit formula}, we have
\begin{multline*}
 S_{dN+1} \left( \delta , \chi \right) =   e^{- 2\pi i \delta (dN+1)} \left( (dN+1) 1_{\chi'} (\chi) - \sum_{\rho \in Z(\chi; T)} \dfrac{(dN+1)^\rho}{\rho} \right) \\ + 2 \pi i \delta \int_{N^{1/8}}^{dN+1}  \left( t 1_{\chi'}(\chi) - \sum_{\rho \in Z(\chi; T)}\dfrac{t^\rho}{\rho} \right) e^{-2\pi i \delta t} dt + O\left( (1+ dN |\delta|) \dfrac{dN\log^2 (dqN)}{T} \right).
 \end{multline*}
The error term above follows from integrating the error term in the explicit formula. By integration by parts, we have
\begin{align*}
&2\pi i \delta \int_{N^{1/8}}^{dN+1}  \left( t 1_{\chi'} (\chi) - \sum_{\rho \in Z(\chi;T)}\dfrac{t^\rho}{\rho} \right) e^{-2\pi i \delta t} dt \\ 
= & - \left[   e^{-2\pi i \delta t} \left( t 1_{\chi'} (\chi) - \sum_{\rho \in Z(\chi;T)}\dfrac{t^\rho}{\rho} \right)\right]_{N^{1/8}}^{dN+1} + \int_{N^{1/8}}^{dN+1}  \left( 1_{\chi'}(\chi)  - \sum_{\rho \in Z(\chi;T)} t^{\rho-1} \right) e^{-2\pi i \delta t} dt \\
= & - e^{- 2\pi i \delta (dN+1)} \left( (dN+1) 1_{\chi'} (\chi) - \sum_{\rho \in Z(\chi; T)}\dfrac{(dN+1)^\rho}{\rho} \right) \\ 
& +  \int_{N^{1/8}}^{dN+1}  \left(1_{\chi'}(\chi)  - \sum_{\rho \in Z(\chi;T)} t^{\rho-1} \right) e^{-2\pi i \delta t} dt  +  O\left( N^{1/8} T \log(dqT) \right),
\end{align*}
where the error term comes from bounding the term $e^{-2\pi i \delta N^{1/8}} ( N^{1/8} 1_{\chi'} (\chi) - \sum_{\rho \in Z(\chi;T)}{N^{\rho/8}}/{\rho} )$ using upper bound on $|Z(\chi; T)|$ (see Iwaniec--Kowalski~\cite[Theorem 5.24, or Principle 2 Chapter 18]{IK04}). The lemma follows from the equations above.

For $\delta = 0$, we have $S_{dN+1}\left( 0 , \chi \right) = \sum_{n \leq dN+1} \Lambda(n) \chi(n)$, and so an application of Proposition~\ref{truncated explicit formula} yields
\begin{multline*}
S_{dN+1}\left( 0 , \chi \right) =  (dN+1) 1_{\chi'}(\chi) - \sum_{\rho \in Z(\chi; T)}\dfrac{(dN+1)^\rho}{\rho} + O \left( \dfrac{dN \log^2 (dqN)}{T} \right) \\
=  \int_{N^{1/8}}^{dN+1}  \left( 1_{\chi'} (\chi) -\sum_{\rho \in Z(\chi; T)} t^{\rho-1} \right)  dt +  O \left( \dfrac{dN \log^2 (dqN)}{T} \right).
\end{multline*}

\end{proof}

\begin{proof}[Proof of Proposition~\ref{major arc estimate}]
Here we prove the proposition for $(D,T)$ exceptional; the unexceptional case can be shown in a similar manner.

By substituting Lemma~\ref{estimate} into Lemma~\ref{rearrange the sum}, we have
\begin{multline*}
\widehat{F_{N,d}}\left( \dfrac{a}{q} + \delta \right) = \dfrac{e^{2\pi i \delta/d}}{\phi(dq)} \int_{N^{1/8}}^{dN} e^{-2 \pi i t \delta /d} \left( G_{a,q,d,\chi'} -
 t^{\beta_D - 1} G_{a,q,d,\chi'\chi_D} \right) dt  \\
- \dfrac{e^{2\pi i \delta/d}}{\phi(dq)} \sum_{\chi \; (\mathrm{mod} \; dq) } \int_{N^{1/8}}^{dN} \left( \sum_{\substack{\rho \in Z(\chi; T) \\ \rho \neq \beta_D}} t^{\rho - 1} \right) e^{-2\pi i t \delta /d} G_{a,q,d,\chi}dt 
+ O \left( (1 + N|\delta| )\dfrac{dqN\log^2 N}{T} \right),
\end{multline*}
where $\chi'$ is the principal character of modulus $dq$.

Since the modulus of $\chi_D$ is a divisor of $d$, we have $\chi_D(dm+1) = 1$ and it follows that
\begin{equation*}
\chi' (dm+1) = \chi'\chi_D (dm+1)
\end{equation*}
for all $0 \leq m \leq q-1$. Therefore, the first integral appearing in the expression for $\widehat{F_{N,d}} (a/q +\delta)$ above is equal to
\begin{equation}\label{exceptional character equality}
\dfrac{e^{2\pi i \delta/d} G_{a,q,d,\chi'}}{\phi(dq)} \int_{N^{1/8}}^{dN} \left( 1-  t^{\beta_D - 1} \right)e^{-2\pi i t \delta /d} dt.
\end{equation}
On the other hand, by the second assertion of Proposition~\ref{Explicit formula}, for sufficiently large $C_3$ and $t > N^{1/8}$ we have
\begin{multline*}
  \dfrac{1}{\phi(dq)} \sum_{\chi \; (\mathrm{mod} \; dq) }  \left| - e^{2\pi i (\delta/d - t\delta/d)} \sum_{ \substack{\rho \in Z(\chi; T) \\ \rho \neq \beta_D}} t^{\rho-1} G_{a,q,d,\chi} \right| \\
\leq  \dfrac{ |G_{a,q,d,\chi}| }{\phi(d)\phi(q)}  \left( \sum_{\chi \; (\mathrm{mod} \; dq) }\sum_{\substack{\rho \in Z(\chi; T) \\ \rho \neq \beta_D}} |t^{\rho-1}| \right)
 \ll \dfrac{q}{\phi(d)\phi(q)} (1 - \beta_D)\log(dqT) \exp\left( - \dfrac{c_{3}}{8} \dfrac{ \log N}{\log(dqT)} \right),
\end{multline*}
where we take the trivial bound $|G_{a,q,d,\chi}| \leq q$ and use the inequality $\phi(dq) \geq \phi(d)\phi(q)$. It follows that
\begin{multline*}
\left| \widehat{F_{N,d}} \left( \dfrac{a}{q} + \delta \right) - \dfrac{e^{2\pi i \delta /d} G_{a,q,d,\chi'}}{\phi(dq)} \int_{N^{1/8}}^{dN} \left( 1- t^{\beta_D -1} \right) e^{-2\pi i t \delta /d} dt \right| \\
\leq  O\left( \dfrac{dqN}{\phi(d)\phi(q)}(1 - \beta_D)\log(dqT) \exp\left( - \dfrac{c_{3}}{8} \dfrac{ \log N}{\log(dqT)} \right) \right) + O \left( (1 + N|\delta| )\dfrac{dqN\log^2 N}{T} \right).
\end{multline*}
By taking $q = 1$ and $\delta = 0$ in the estimates above, we can deduce that
\begin{equation}\label{FT at zero}
\left| \widehat{F_{N,d}} (0) -  \dfrac{1}{\phi(d)} \int_{N^{1/8}}^{dN} \left( 1- t^{\beta_D -1} \right) dt \right| \leq \dfrac{dN}{\phi(d)} \dfrac{(1 - \beta_D)\log (dT)}{4c_1}
+ O \left( \dfrac{dN\log^2 N}{T} \right),
\end{equation}
where we used the assumption $N \geq (DT)^{C_3}$ to obtain the first term on the right hand side; recall that $c_1$ is the constant involved in Definition~\ref{dichotomy}. Thus, by the triangle inequality and (\ref{Gauss sum principal character}), we can conclude that
\begin{multline*}
\left| \widehat{F_{N,d}} \left( \dfrac{a}{q} + \delta \right)\right| \leq \dfrac{ |\widehat{F_{N,d}}(0)| }{\phi(q)} + \dfrac{dN}{\phi(d)\phi(q)} \dfrac{(1 - \beta_D)\log (dT)}{4c_1} \\
+  O\left(\dfrac{dqN}{\phi(d)\phi(q)} (1 - \beta_D)\log(dqT) \exp\left( -\dfrac{c_{3}}{8} \dfrac{ \log x}{\log(dqT)} \right) \right) + O \left( (1 + N|\delta| )\dfrac{dqN\log^2 N}{T} \right).
\end{multline*}
Therefore, we can deduce the first inequality if we manage to prove the second one.

In order to prove the lower bound on $|\widehat{F_{N,d}}(0)|$, we need to bound the integral
$$\dfrac{1}{\phi(d)} \int_{N^{1/8}}^{dN} \left( 1- t^{\beta_D -1} \right) dt$$
from below. To do so we use the inequality $1 - e^{-x} \geq x/(x+1)$ which holds\footnote{This can be verified by taking second order derivatives of $e^x$ and $x+1$.} for all $x >0$. For $t \geq N^{1/8}$, one has $t \geq (DT)^{C_3/8} \geq (dqT)^{C_3/(8C_1)}$, and so
$$1 - t^{\beta_D - 1} \geq 1 - (dqT)^{-C_3(1-\beta_D )/ (8C_1)} \geq  \dfrac{C_3 (1 - \beta_D) \log (dqT)}{8C_1 + C_3(1-\beta_D) \log (dqT)}.$$
Since $1 - \beta_D \leq c_1 / (C_1 \log(DT)) \leq c_1 / \log (dq T)$ and $C_3$ is sufficiently large, we have
$$1 - t^{\beta_D - 1} \geq  \dfrac{C_3(1 - \beta_D) \log (dqT)}{8C_1 + C_3 c_1} \geq \dfrac{1 - \beta_D}{2c_1} \log(dqT).$$
Thus, 
\begin{equation}\label{lower bound - main contribution}
\dfrac{1}{\phi(d) } \int_{N^{1/8}}^{dN} \left( 1-  t^{\beta_D - 1} \right) dt \geq \dfrac{dN}{\phi(d)} \dfrac{(1- \beta_D)\log(dT)}{2c_1} - O(N^{1/8}).
\end{equation}
Therefore, by the triangle inequality in (\ref{FT at zero}),
\begin{multline*}
\left| \widehat{F_{N,d}} (0)  \right| \geq  \dfrac{1}{\phi(d)} \int_{N^{1/8}}^{dN} \left( 1- t^{\beta_D -1} \right) dt  - \dfrac{dN}{\phi(d)} \dfrac{(1 - \beta_D)\log (dT)}{4c_1}
- O \left( \dfrac{dN\log^2 N}{T} \right)\\
\geq  \dfrac{dN}{\phi(d)} \dfrac{(1- \beta_D)\log(dT)}{4c_1} - O \left( \dfrac{dN\log^2 N}{T} \right),
\end{multline*}
as claimed.
\end{proof}

\section{Proving the main result}\label{main argument section}
The main lemma used to prove Theorem~\ref{main result} is an analogue of the main iteration lemma given in Ruzsa and Sanders~\cite[Lemma 8.1]{RS08}. Under certain restrictions on several input parameters, the iteration lemma allows one to find a denser subset, located on a sub-progression, given a set whose difference set does not contain certain affine transformations of primes. 

One can then apply the iteration lemma and conclude that one of the hypotheses must fail after sufficiently many iteration steps, since otherwise the density increment would lead to a subset with density larger than $1$ which is impossible. The occurrence of the restriction hypotheses implies that either the difference set contains an element which is one less than a prime, or the desired upper bound on the density holds.

\begin{lemma}\label{iteration}
There exist positive absolute constants $C_4, C_5, c_5, c_6, c_7, c_8$ such that we can obtain the following result.\footnote{We need to introduce an upper bound on $N$ due to the factor $(\log N)^4$ in the minor arc estimate~\cite[Corollary 6.2]{RS08}.} 

Let $D \geq 2$ and let $N$ be a positive integer such that $D^{C_4} < N \leq \exp(D^{1/10})$. Let $T = D^{C_1^2}$. Let $A \subseteq [N]$ have density $\alpha >0$.

Let $d$ be a positive integer and assume one of the following:
\begin{enumerate}[(a)]
\item $(D,T)$ is unexceptional, $d\alpha^{-1} \leq c_{5} D^{c_{5}}$;
\item $(D,T)$ is exceptional, $d$ is a multiple of $d_D$, and $d\alpha^{-1} \leq c_{5} D^{1 + c_{5}}$.
\end{enumerate}
Suppose that $A - A$ does not contain any number which can be written as $(p-1)/d$ for some prime number $p$. Suppose also that
$$\log N \geq C_5 (\log \alpha^{-1}+ \log_3 D + 1) (\log D + \log_2 N + 1).$$
Then there exists a positive integer $d'$ with $d' \leq c_6 \alpha^{-3}$ and a progression $P'$ with common difference $d'$ and length $\geq (c_7 \alpha / d \log N)^8 N$ such that $|A \cap P'| \geq \alpha (1 + c_8) |P'|.$
\end{lemma}

We follow a similar strategy to prove Lemma~\ref{iteration} as Ruzsa and Sanders~\cite[Section 8]{RS08}. The idea used to obtain density increment there is energy increment. To proceed, we first notice that by the same argument as theirs, we can conclude that there exist absolute constants $c_9, c_{10}$ such that for any $N, D, d, T, A, \alpha$ satisfying the same hypotheses as Lemma~\ref{iteration}, we can deduce the following. By choosing
\begin{equation}\label{parameters}
N' :=  \lfloor c_9 \alpha N \rfloor, \; Q' := \dfrac{d^4 \log^8 N'}{c_{10}^2 \alpha^2} \;\text{and} \; Q := \dfrac{N'}{Q'}
\end{equation}
and taking the major arcs to be
$$\mfr{M} := \bigcup_{q \leq Q'} \mfr{M}^*_q,$$
where $\mfr{M}_q^*$ and $\mfr{M}_{a,q}$ (involved in the definition of $\mfr{M}_q^*$) are defined with respect to $Q$ as in (\ref{major arcs notation}), we have
\begin{equation}\label{lower bound at 0}
|\widehat{F_{N',d}} (0) | \gg \dfrac{N'}{d}
\end{equation}
and 
\begin{equation}\label{total energy - major arcs}
\int_{ \mfr{M}} \left| \left(\widehat{1_A} - \alpha \widehat{1_I} \right)  (\theta) \right|^2 |\widehat{F_{N', d}} (\theta) | d\theta \gg \alpha^2 N  |\widehat{F_{N',d}}(0)|.
\end{equation}

The inequality (\ref{lower bound at 0}) is a consequence of Proposition~\ref{major arc estimate} and an upper bound on the size of the exceptional zero (see Iwaniec and Kowalski~\cite[Theorem 5.28]{IK04}). There are two steps towards obtaining (\ref{total energy - major arcs}): the first is to show that 
$$\int_{\mb{T}}  \left| \left(\widehat{1_A} - \alpha \widehat{1_I} \right) (\theta) \right|^2 |\widehat{F_{N', d}} (\theta) | d\theta \gg \alpha^2 N |\widehat{F_{N',d}} (0)|,$$
and the second is to use the minor arc estimate (see Ruzsa and Sanders~\cite[Section 6]{RS08}) to bound the integral on the minor arcs $\mfr{m} := \mb{T} \setminus \mfr{M}$.

Thus, to prove Lemma~\ref{iteration}, it suffices to obtain estimates needed for the Ruzsa--Sanders method~\cite[Corollary 7.3, Section 8]{RS08} on the major arcs, which are given in Lemma~\ref{L^2 concentration} below.

\begin{lemma}\label{L^2 concentration}
There exists a positive absolute constant $C_6$ such that for any $N, D, d, T, A, \alpha$ satisfying the same hypotheses as Lemma~\ref{iteration} and $Q, N'$ as defined in (\ref{parameters}), we have
$$\sup_{\theta \in \mfr{M}^*_q} |\widehat{F_{N', d}} (\theta) |  \ll \dfrac{{|\widehat{F_{N', d}} (0)|}}{\phi(q)} \text{ for all } q \leq C_6 \alpha^{-3},$$
and
$$\sum_{q \leq C_6 \alpha^{-3}} \int_{\mfr{M}^*_q}  \left| \left(\widehat{1_A} - \alpha \widehat{1_I} \right)  (\theta) \right|^2 |\widehat{F_{N', d}} (\theta) | d\theta \gg \alpha^2 N |\widehat{F_{N',d}} (0)|.$$
\end{lemma}

\begin{proof}
By Dirichlet's pigeonhole principle, we have
$$\mfr{M} = \mfr{M}_1 \cup \mfr{M}_2,$$ 
where
$$\mfr{M}_1 := \bigcup_{q \leq C_6 \alpha^{-3}} \mfr{M}^*_q \text{ and } \mfr{M}_2 := \bigcup_{C_6 \alpha^{-3} < q \leq Q'} \mfr{M}^*_q.$$
We can employ our major arc estimates on both $\mfr{M}_1$ and $\mfr{M}_2$, since the relevant hypotheses required by Proposition~\ref{major arc estimate} are satisfied as long as $c_5$ is sufficiently small and $C_5$ is large. Since $c_5 < 1$ and $C_1 \geq 10$, it follows that for all $dq \leq D^{C_1}$ and $a/q + \delta \in \mfr{M}_q^*$, one has
$$(1 + N'|\delta| )\dfrac{dqN'\log^2 N'}{T} \ll \dfrac{d^5 q N' \log^{10} N'}{\alpha^2 D^{C_1^2}} \ll \dfrac{N'}{D^{2C_1}}.$$
Thus, irrespective of whether $(D,T)$ is exceptional or not, for any $a/q + \delta \in \mfr{M}_q^*$ where $q \leq Q'$ one has
\begin{equation}\label{major arc - 2}
\left| \widehat{F_{N',d}} \left( \dfrac{a}{q} + \delta \right) \right| \leq \dfrac{2 |\widehat{F_{N',d}} (0)| }{\phi(q)}+ 
O\left( \dfrac{| \widehat{F_{N',d}}(0)|q\log(dqT)}{\phi(q)\log(dT)} \exp\left( - \dfrac{c_4 \log N'}{\log(D^{C_1} T)} \right) \right) \\
+ O \left( \dfrac{N'}{D^{2C_1}} \right).
\end{equation}
By (\ref{lower bound at 0}) and $dq \leq D^{C_1}$, we always have
\begin{equation}\label{tiny error}
\dfrac{2 |\widehat{F_{N',d}} (0)| }{\phi(q)}+  O \left( \dfrac{N'}{D^{2C_1}} \right) \geq \dfrac{|\widehat{F_{N',d}} (0)| }{\phi(q)}.
\end{equation}
To deal with the second term, notice that since $dq \leq D^{C_1}$, $T = D^{C_1^2}$, and $C_5$ is sufficiently large, we have
\begin{equation}\label{bounding the second term}
\dfrac{|\widehat{F_{N',d}}(0)| q\log(dqT)}{\phi(q)\log(dT)} \exp\left(- c_4 \dfrac{\log N'}{\log (D^{C_1} T)} \right) \leq C_6^{-1} \alpha^3 |\widehat{F_{N',d}}(0)|.
\end{equation}
Since $\alpha^{3} \leq \alpha $, it follows from (\ref{major arc - 2}), (\ref{tiny error}) and (\ref{bounding the second term}) that for all $C_6 \alpha^{-3} < q \leq Q'$, we have
\begin{equation}\label{sub-critical range}
\sup_{\theta \in \mfr{M}^*_{q} } |\widehat{F_{N', d}} (\theta) | \leq C_6^{-1} \alpha^3 {|\widehat{F_{N', d}} (0)|} \ll C_6^{-1} \alpha{|\widehat{F_{N', d}} (0)|}.
\end{equation}
Since $\alpha^3 \ll \min_{q \leq  C_6 \alpha^{-3}}\{ 1/ \phi(q) \},$ it follows from (\ref{major arc - 2}), (\ref{tiny error}) and (\ref{bounding the second term}) that
\begin{equation}\label{critical range}
\sup_{\theta \in \mfr{M}^*_q} |\widehat{F_{N', d}} (\theta) |  \ll \dfrac{{|\widehat{F_{N', d}} (0)|}}{\phi(q)} \text{ for all } q \leq C_6 \alpha^{-3},
\end{equation}
which is the first assertion.

By substituting (\ref{sub-critical range}) and applying Plancherel's theorem, we have
\begin{equation}\label{energy - sub-critical range}
\int_{\mfr{M}_2 } \left| \left(\widehat{1_A} - \alpha \widehat{1_I} \right) (\theta) \right|^2 |\widehat{F_{N', d}} (\theta)| d\theta \ll C_6^{-1} \alpha^2 N |\widehat{F_{N', d}} (0) |.
\end{equation}
Therefore, combining the lower bound obtained in (\ref{total energy - major arcs}) and the upper bound (\ref{energy - sub-critical range}), there exists a large absolute constant $C_6$ so that
\begin{equation}\label{critical range - energy}
 \int_{\mfr{M}_1} \left| \left(\widehat{1_A} - \alpha \widehat{1_I} \right)  (\theta) \right|^2 |\widehat{F_{N', d}} (\theta)| d\theta \gg \alpha^2 N |\widehat{F_{N', d}} (0) |.
 \end{equation}
The second assertion follows from (\ref{critical range - energy}) and the triangle inequality.

\end{proof}

We note that our main modification is the size bound on $N$ introduced in the assumption of the iteration lemma. More specifically, we obtain the same density increment under the weaker condition $\log N \gg (\log \alpha^{-1}+ \log_3 D +1) (\log D + \log_2 N + 1)$, which is $\log N \gg (\log D)^2$ for Ruzsa and Sanders. This strengthening leads to our improvement.

\begin{proof}[Proof of Theorem~\ref{main result}]
Let $C'$ be a sufficiently large constant and let 
$$D := \exp\left( \dfrac{\log N}{C' (\log \alpha^{-1} + \log_2 N +1)} \right) \text{  and  }T := D^{C_1^2}.$$
The result follows from applying the same argument as Ruzsa and Sanders~\cite[Proof of Theorem 1.1]{RS08} with $D_0 = D$ and $D_1 = D^{C_1}.$
\end{proof}

\section*{acknowledgement}
The author would like to thank Tom Sanders for his supervision, James Maynard and Joni Ter\"av\"ainen for discussions, and an anonymous referee for suggestions.

\end{document}